\documentclass[12pt,a4paper]{article}
\usepackage{amssymb,amsfonts}
\usepackage{latexsym}
\usepackage{amsmath,amsthm}

\theoremstyle{plain}
\newtheorem{theorem}{Theorem}

\newtheorem{lemma}{Lemma}
\newtheorem{proposition}{Proposition}

\theoremstyle{remark}

\theoremstyle{definition}
\newtheorem{definition}{Definition}[section]

\newcommand{\R}{\mathbb R}
\newcommand{\E}{\mathbb R^n}
\newcommand{\f}{f:X\to\R}

\newcommand{\be}{\begin{equation}}
\newcommand{\ee}{\end{equation}}
\newcommand{\ld}[1]{#1^{\prime}_{-}}

\newcommand{\ph}{\varphi}

\title{ Characterizations of Pseudoconvex Functions}
\author{Vsevolod I. Ivanov\thanks{Email: vsevolodivanov@yahoo.com
\vspace{6pt} }
\\\vspace{6pt}{\em{\small Department of Mathematics, Technical University of Varna, 9010 Varna, Bulgaria} }
}

\begin{document}
\maketitle

\begin{abstract}
A differentiable function is pseudoconvex if and only if its restrictions over straight lines are pseudoconvex. A differentiable function depending on one variable, defined on some closed interval $[a,b]$ is pseudoconvex if and only if there exist some numbers $\alpha$ and $\beta$ such that $a\le\alpha\le\beta\le b$ and the function is strictly monotone decreasing on $[a,\alpha]$, it is constant on $[\alpha,\beta]$, the function is strictly monotone increasing on $[\beta,b]$ and there is no stationary points outside $[\alpha,\beta]$. This property is very simple. In this paper, we show that a similar result holds for lower semicontinuous functions, which are pseudoconvex with respect to the lower Dini derivative. We prove that a function, defined on some interval, is pseudoconvex if and only if its domain can be split into three parts such that the function is strictly monotone decreasing in the first part, constant in the second one, strictly monotone increasing in the third part, and every stationary point is a global minimizer. Each one or two of these parts may be empty or degenerate into a single point. Some applications are derived.



\end{abstract}

\section{Introduction}
Pseudoconvex functions are among the most useful generalized convex ones. They were introduced independently by Tuy \cite{tuy64} and Mangasarian \cite{man65}. Their importance for optimization were firstly expressed in the book \cite{man69}.

The behavior of the functions involved is very important for studying optimization problems. Therefore, it is useful  to have some characterizations of the graphs of the functions. A Fr\'echet differentiable function, defined on the Euclidean space $\R^n$, is pseudoconvex if and only if its restrictions over the straight lines are pseudoconvex. This leads us to derive characterizations of the graph of pseudoconvex functions, defined on $\R$. It is easy to see that a differentiable function  is pseudoconvex on some closed interval $[a,b]$ if and only if there exist some numbers $\alpha$ and $\beta$ such that $a\le\alpha\le\beta\le b$ and the function attains its global minimum over $[a,b]$ on the interval $[\alpha,\beta]$, the function is strictly decreasing over $[a,\alpha]$ and strictly increasing over $[\beta,b]$ and there is no stationary points outside $[\alpha,\beta]$. The proof of this complete characterization remains simple even if the function is locally Lipschitz pseudoconvex with respect to the Clarke generalized directional derivative. On the other hand, its proof is not so simple if the function is lower semicontinuous.

In this  paper, we generalize the characterization to lower semicontinuous functions, which are pseudoconvex with respect to the lower Dini directional derivative. Some applications of the characterization are provided.

\section{Complete characterization of pseudoconvex functions }
\label{s2}
\setcounter{theorem}{0}
\setcounter{definition}{0}
\setcounter{lemma}{0}
\setcounter{remark}{0}
Let $X\subseteq\R^n$ be an open set. Then
the lower Dini directional derivative of the function $f:X\to\R$ at the point $x$ in direction $u$ is defined as follows: 

\begin{equation}\label{0.0}
\ld f(x,u)=\liminf_{t\to +0}\, t^{-1}[f(x+tu)-f(x)]\,,
\end{equation}

\begin{definition}
A function $f$ defined on some convex set $X\subseteq\R^n$ is called (strictly) pseudoconvex with respect to the lower Dini directional derivative iff the following implication holds:
\[\begin{array}{rl}
x\in X,\; y\in X,\; f(y)<f(x)\quad & \Rightarrow\quad \ld f(x,y-x)<0 \\
(x\in X,\; y\in X,\; y\ne x,\; f(y)\le f(x)\quad & \Rightarrow\quad \ld f(x,y-x)<0).
\end{array}\]
\end{definition}

\begin{proposition}
Let $X\subseteq\R^n$ be a nonempty convex set and $\f$ be a given function. Then $f$ is (strictly) pseudoconvex with respect to the lower Dini derivative if and only if for all $x$, $y\in X$ the restriction of $f$ over straight lines  $\ph:X(x,y)\to\R$, where 
\[
X(x,y)=\{t\in\R\mid x+t(y-x)\in X\},
\]
defined by the equality $\ph(t)=f(x+t(y-x))$, is (strictly) pseudoconvex.
\end{proposition}
\begin{proof}
The proof is easy and it follows from the fact that the lower Dini derivative is radial.
\end{proof}

It is said that a function $\ph$, depending on one variable, is increasing (decreasing) on the interval $I$ iff $\ph(t_2)\le\ph(t_1)$   
($\ph(t_2)\ge\ph(t_1)$)  for all $t_1\in I$, $t_2\in I$ such that $t_2<t_1$. It is said that a function $\ph$ is strictly increasing (decreasing) on the interval $I$ iff $\ph(t_2)<\ph(t_1)$ ($\ph(t_2)<\ph(t_1)$) for all $t_1\in I$, $t_2\in I$ such that $t_2<t_1$.

\begin{lemma}\label{lema1}
Let $\ph:I\to\R$, where $I\subset\R$ is an interval, be lower semicontinuous pseudoconvex function. Suppose that $a$, $b\in I$, be given numbers such that $a<b$, $\ph(a)<\ph(b)$. Then there exists a number $s<b$ with the property that $\ph$ is strictly monotone increasing over $I\cap [s,+\infty)$. In particular, $\ph$ is strictly monotone increasing over $I\cap [b,+\infty)$.
\end{lemma}
\begin{proof}
Consider the case when $b$ is an interior point of the interval $I$. Then we obtain from the lower semicontinuity of $\ph$ and by the inequality  $\ph(a)<\ph(b)$ that there exist numbers $s_1$, $s_2\in I$ such that $s_1<b<s_2$ and
\[
\ph(a)<\ph(t)\quad\forall t\in (s_1,s_2).
\]

\smallskip
 1) {\em We prove that $\ph$ is strictly monotone increasing over the open interval $(s_1,s_2)$.}
Indeed, suppose  the contrary that there exist such numbers $t_1$, $t_2\in (s_1,s_2)$ that $t_1<t_2$ and $\ph(t_1)\ge\ph(t_2)$. Consider the function
\[
\psi(t)=(t_2-t_1)\ph(t)+(t-t_2)\ph(t_1)+(t_1-t)\ph(t_2),\quad t_1\le t\le t_2.
\]
It is lower semicontinuous and $\psi(t_1)=\psi(t_2)=0$. By Weierstrass extreme value theorem the function $\psi$ attains its minimal value on the closed interval $[t_1,t_2]$ at some point $\xi$, $\xi>t_1$. According to the minimality of $\xi$, we have $\psi[\xi+t(t_1-t_2)]-\psi(\xi)\ge 0$ for all sufficiently small positive numbers $t$. 
Therefore $\ld\psi(\xi,t_1-t_2)\ge 0$. We obtain from here that
\[
\ld \ph(\xi,t_1-t_2)\ge\ph(t_1)-\ph(t_2)\ge 0.
\]
It follows from the inequality $\ph(a)<\ph(\xi)$ and pseudoconvexity of $\ph$ that $\ld\ph(\xi,a-\xi)<0$, which is a contradiction because of the positive homogeneity of Dini derivative. 
Thus $\ph$ is strictly monotone increasing over $(s_1,s_2)$.

\smallskip
2) We can suppose without loss of generality that $(s_1,s_2)$ is the maximally wide on the right open interval such that $\ph$ is strictly monotone increasing. Really the left end $s_1$ is not essential, but $s_2=\sup\{c\mid\ph \;\textrm{is strictly monotone increasing over}\;(s_1,c)\}$. If $s_2$ is the endpoint of $I$, then the function will be strictly monotone increasing over
$(int\, I)\cap [s_1,+\infty)$ and $s=s_1$. Consider the case when $s_2$ is an interior point of $I$ and $\ph(a)<\ph(s_2)$. Repeating the reasoning above and taking $s_2$ instead of $b$, we can increase the interval $(s_1,s_2)$, keeping the property of strict monotonicity, which contradicts the assumption. Therefore, this case is impossible. It remains to consider the case when $s_2$ is an interior point for $I$ and  $\ph(s_2)\le\ph(a)$. Then it follows from $\ph(a)<\ph(b)$ that $\ph(s_2)<\ph(b)$. By pseudoconvexity, we can make the conclusion that 
$\ld \ph(b,s_2-b)<0$.
On the other hand, taking into account that $\ph$ is strictly monotone over $(s_1,s_2)$, we obtain that 
\[
\ph(b+t(s_2-b))-\ph(b)>0
\]
 for all sufficiently small positive numbers $t$. Then
\[
\ld \ph(b,s_2-b)=\liminf_{t\to+0}\,(1/t)\,\left(\ph(b+t(s_2-b))-\ph(b)\right)\ge 0,
\]
which contradicts our assumption. Therefore, this case is also impossible.

\smallskip
3) Thus, we proved that $\ph$ is strictly monotone increasing over $({\rm int}\, I)\cap [s_1,+\infty)$. We prove that  if $s_2\in I$, then $\ph(t)<\ph(s_2)$ for all $t$ with $s_1<t<s_2$. Assume the contrary that there exists 
$\alpha\in(s_1,s_2)$ such that $\ph(s_2)\le\ph(\alpha)$. Since the interval $(s_1,s_2)$ is open and $\ph$ is strictly monotone increasing over it, then there exists such a number $\beta>\alpha$ that $\ph(\alpha)<\ph(\beta)$, $\beta\in(\alpha,s_2)$. Therefore $\ph(s_2)<\ph(\beta)$. According to pseudoconvexity $\ld \ph(\beta,s_2-\beta)<0$. On the other hand,
\[
\ld \ph(\beta,s_2-\beta)=\liminf_{t\to+0}\,(1/t)\,\left(\ph(\beta+t(s_2-\beta))-\ph(\beta)\right)\ge 0,
\]
which is a contradiction.

\smallskip
4) At last, if $b$ is an endpoint of $I$, then repeating the reasoning above, we can see that the lemma is also satisfied in this case.
\end{proof}

\begin{lemma}\label{lema2}
Let $\ph:I\to\R$ be a lower semicontinuous pseudoconvex function, which is defined on some interval $I\subset\R$.
Denote by $\hat I$ the set of points, where the global minimum of $\ph$ over $I$ is attained. Then $\hat I=\emptyset$, in which case $\ph$ is strictly monotone over $I$, or $\hat I$ is an interval, eventually containing only one point, in which case $\ph$ is constant over $\hat I$, strictly monotone increasing over  $I_+=\{t\in I\mid t\ge\hat t \mbox{ for all } \hat t\in\hat I\}$ and strictly monotone decreasing over
 $I_-=\{t\in I\mid t\le\hat t \mbox{ for all } \hat t\in\hat I\}$.
\end{lemma}
\begin{proof}
1) {\em The set $\hat I$ is empty or it is an interval, which eventually consists of an unique point.}
We can prove this showing that the following condition is satisfied for arbitrary $t_-<t_+$ with $t_-,\,t_+\in\hat I$: $[t_-,\,t_+]\subset \hat I$. If this not the case, then there is a point $t_0$, $t_-< t_0< t_+$ such that $\ph(t_-)<\ph(t_0)$. According to Lemma \ref{lema1}, $\ph$ is strictly monotone increasing over $I\cap [t_0,\,+\infty)$. Therefore,  $\ph(t_-)<\ph(t_0)<\ph(t_+)$. On the other hand, since both $t_-$ and $t_+$ are points, where the global minimum of  $\ph$ is reached, then the following holds: $\ph(t_-)=\ph(t_+)$. Thus, we obtained a contradiction.

\smallskip
2) {\em If $\hat I\neq\emptyset$, then $\ph$ is strictly monotone increasing over $I_+$ and strictly monotone decreasing over $I_-$}. We prove the first claim. The cases when $I_+=\emptyset$ or $I_+$ consists of an unique point are trivial. Let us suppose now that $I_+$ is not an empty set, nor consists of an unique point. Let $t_-\in\hat I$ and $t_+\in I_+$ are such that $t_+>\inf I_+$. Then $t_+\notin \hat I$. Otherwise the whole interval $[t_-,\,t_+]$ will be included in $\hat I$ and the inequality $t_+>\inf I_+$ will be not satisfied.
Hence $\ph(t_-)<\ph(t_+)$. We conclude from here, according to Lemma \ref{lema1} that $\ph$ is strictly monotone increasing over the interval
 $I\cap [t_+,\,+\infty)$. Since the strict monotonicity is satisfied for every $t_+>\inf I_+$  and $\ph$ upper semicontinuous, then it is easy to see that $\ph$ is strictly increasing over $I_+$.  Using similar arguments we can prove that  $a$, $b\in I$, $a<b$, $\ph(a)>\ph(b)$ implies that the function is strictly monotone decreasing over $I\cap (-\infty,a]$.  It follows from here that $\ph$ is strictly monotone decreasing over $I_-$.

\smallskip
3) {\em If $\hat I=\emptyset$, then $\ph$ is strictly monotone.}
Let us choose a sequence $\{t_n\}_{n=0}^\infty$, which consists of points from $I$ with the property $\ph(t_0)>\ph(t_1)> \dots > \ph(t_n) > \dots$ and  $\lim_n \ph(t_n)=\inf \{\ph(t)\mid t\in I\}$.
Suppose that $t_0>t_1$. We prove then that $\{t_n\}$ is strictly monotone decreasing and the function $\ph$ is strictly increasing over $I$.
To prove that $\{t_n\}$ is strictly decreasing, it is enough to show that $t_2< t_1$ (analogously $t_2< t_1$ implies that $t_3< t_2$ and so on).
Indeed, we have $t_2\notin [t_0, +\infty)$, because for $t\in [t_0, +\infty)$ is satisfied the inequality $\ph(t_0)\le\ph(t)$.
The inequality $\ph(t_1)<\ph(t_0)$ implies that $\ph$ is strictly increasing over $I\cap [t_0, +\infty)$. We also have $t_2\notin [t_1,\,t_0]$. Otherwise, if $t_1<t_2<t_0$ it follows from $\ph(t_1)>\ph(t_2)<\ph(t_0)$ that $\ph$ is strictly decreasing over $I\cap (-\infty,\,t_1]$ and strictly increasing over $I\cap [t_0, +\infty)$. Therefore, 
\[
\hat I=\{\xi\in [t_1,\,t_0] \mid \ph(\xi)=\min_{t_1\le t\le t_0} \ph(t) \}\ne\emptyset,
\]
since every lower semicontinuous function attains its minimal value over any closed interval by Weierstrass theorem. This fact contradicts the assumption $\hat I=\emptyset$.

The inequality $\ph(t_{n+1})<\ph(t_n)$ implies that $\ph$ is strictly increasing over the interval $I\cap [t_n, +\infty)$; therefore it is strictly increasing over the set $J=\bigcup_{n=1}^\infty \left( I\cap [t_n, +\infty) \right)$.
It remains to prove that $J=I$. If this not the case, then there exists a point $t_-\in I\setminus J$. Therefore $t_-< t_n$ for $n=0,\,1,\,\dots\,$.
The following inequality is satisfied for every $n$: $\ph(t_-)\le\ph(t_n)$. Otherwise, there is an integer $n$ such that $\ph(t_-)>\ph(t_{n+1})<\ph(t_n)$, which leads us to a contradiction to the assumption $\hat I=\emptyset$.
We have also that $\ph(t_-)\le\ph(t_n)$ for every $n$ implies that $\ph(t_-)\le\inf\{\ph(t)\mid t\in I\}$,
Therefore $t_-\in\hat I$, which contradicts to the condition  $\hat I=\emptyset$.
\end{proof}

\begin{theorem}\label{th1}
Let $\ph:I\to\R$ be a lower semicontinuous function, which is defined on some interval $I\subset\R$.
Denote by $\hat I$ the set of points, where the global minimum of $\ph$ over $I$ is attained. Then $\ph$ is pseudoconvex with respect to the lower Dini derivative if and only if the following conditions are satisfied:

\begin{enumerate}
\item[(i)]
$\hat I=\emptyset$, in which case $\ph$ is strictly monotone over $I$, or $\hat I$ is an interval, eventually containing only one point, in which case $\ph$ is constant over $\hat I$, strictly monotone increasing over  $I_+=\{t\in I\mid t\ge\hat t \mbox{ for all } \hat t\in\hat I\}$ and strictly monotone decreasing over  $I_-=\{t\in I\mid t\le\hat t \mbox{ for all } \hat t\in\hat I\}$;
\item[(ii)]
The function has no stationary points in the set $I\setminus\hat I$.
\end{enumerate}
\end{theorem}
\begin{proof}
Let the function be pseudoconvex. Condition (i) is satisfied according to Lemmas \ref{lema1} and \ref{lema2}. 
We prove (ii). Assume the contrary that there is $t_1\in I\setminus\hat I$, which is a stationary point. If $\hat I\ne\emptyset$, then it follows from $t_1\notin\hat I$ that there exists $t_2\in I$ with $\ph(t_2)<\ph(t_1)$. If $\hat I=\emptyset$, then $\ph$ is monotone on $I$. In the case when $\ph$ is monotone increasing, then the interval $I$ is not closed on the left, because $\hat I=\emptyset$. Therefore, again there exists $t_2\in I$ with $\ph(t_2)<\ph(t_1)$. The case when $\ph$ is monotone decreasing is similar.  By pseudoconvexity, we conclude that $\ld \ph(t_1,t_2-t_1)<0$, which contradicts the stationarity of $t_1$.

Let conditions (i) and (ii) be satisfied. We prove that $\ph$ is pseudoconvex. Choose point $t_1\in I$ and $t_2\in I$ such that 
$\ph(t_2)<\ph(t_1)$. It follows from here that $t_1\notin\hat I$ or $\hat I=\emptyset$. If $t_1\notin\hat I$, then two cases are possible: $t_1\in I_+$ or $t_1\in I_-$. If $t_1\in I_+$, then $t_1+t(t_2-t_1)\in I_+$ for all sufficiently small $t>0$. Therefore
$\ph[t_1+t(t_2-t_1)]<\ph(t_1)$, and  $\ld \ph(t_1,t_2-t_1)\le 0$. Since $t_1$ is not stationary, then the last inequality must be strict, that is $\ld \ph(t_1,t_2-t_1)<0$. 
The cases $t_1\in I_-$ or $\hat I=\emptyset$, $\ph$ is monotone increasing or $\hat I=\emptyset$, $\ph$ is monotone decreasing
are similar. 
It follows from all cases that $\ph$ is pseudoconvex.
\end{proof}

The proof of the following theorem is similar to the proof of Theorem \ref{th1}:

\begin{theorem}\label{th5}
Let $\ph:I\to\R$ be a lower semicontinuous function, which is defined on some interval $I\subset\R$.
Then $\ph$ is strictly pseudoconvex with respect to the lower Dini derivative if and only if the following conditions are satisfied:
\begin{enumerate}
\item[(i)]
$\hat I=\emptyset$, in which case $\ph$ is strictly monotone over $I$, or $\hat I$ contains only one point $\hat t$,  strictly monotone increasing over  $I_+=\{t\in I\mid t\ge\hat t \}$ and strictly monotone decreasing over  $I_-=\{t\in I\mid t\le\hat t \}$;
\item[(ii)]
The function has no stationary points $t\ne \hat t$.
\end{enumerate}
\end{theorem}

\section{Applications of the characterization}

In this section, we apply Theorems \ref{th1} and \ref{th5} to prove some connections between several types of generalized convex functions in the nondifferentiable case. The proofs in the differentiable case are easy.

\begin{definition}
Let $X\subseteq\E$ be a convex set. Then a function $f:X\to\R$ is called:
\begin{enumerate}
\item [(i)]  quasiconvex over $X$, iff 
 \[ 
		f(x+t(y-x))\le\max(f(x),f(y))\quad \forall x\in X\quad \forall y\in X\quad
\forall t\in[0,1];
\] 
\item [(ii)]  semistrictly quasiconvex over $X$, iff 
\[
   x,\; y\in X,\; f(y)<f(x)\quad\textrm{imply that}\quad
f(x+t(y-x))<f(x)\quad \forall t\in(0,1).
\]
\end{enumerate}
\end{definition}

The following result is well-known and its  proof can be seen in the book of Martos \cite[Theorem 3.21, p. 51]{mar75}:

\begin{lemma}\label{lema4}
A scalar function $f$ defined on a convex set $X$ is semistrictly quasiconvex if and only if any segment $[x,y]\subset X$ can be split into three segments (each dividing point belonging to at least one of the segments it closes) such that $f$ is decreasing in the first, constant in the second, and strictly increasing in the third. Any one or two of these segments may be empty or degenerate into one point.
\end{lemma}

The following theorem was proved in \cite{die81}. Our proof is new.

\begin{theorem}\label{th3}
Let $f:\,X\to\R$ be radially lower semicontinuous pseudoconvex function defined on the convex set $X$. Then $f$ is semistrictly quasiconvex, and moreover $f$ is quasiconvex.
\end{theorem}
\begin{proof}
It obviously follows from Lemma \ref{lema4} and Theorem \ref{th1} that every pseudoconvex function is semistrictly quasiconvex. On the other hand it is well-known that a lower semicontinuous semistrictly quasiconvex function is quasiconvex.
\end{proof}

We introduce several notations from mathematical logic. Denote the negative of the statement $A$ by $\neg A$, the conjunction of the statements $A$ and $B$ by $A \land B$, the disjunction of $A$ and $B$ by $A \lor B$. 

Consider the following statements:
\[
\begin{array}{l}
A:=\{x \textrm{ is a stationary point of the function } f\}; \\
B:=\{ t=0 \textrm{ is a global minimizer of the function } \ph \textrm{ over } X(x,y)\}; \\
C:=\{ t=0 \textrm{ is a stationary point of the function } \ph \textrm{ over } X(x,y)\}.
\end{array}
\]

\begin{lemma}\label{lema3}
Let $f:\,X\to\R$ be radially upper semicontinuous quasiconvex function defined on the convex set $X\subseteq\R^n$. For arbitrary chosen $x\in X$ and $y\in X$ consider the restriction $\ph: X(x,y)\to\R$ of $f$ such that $\ph(t)=f[z(t)]$, where 
\[
z(t)=(1-t)x+ty\quad\textrm{and}\quad X(x,y)=\{t\in\R\mid z(t)\in X\}.
\]
Then $\neg A \land \neg B \Rightarrow \neg C$.

\end{lemma} 
\begin{proof}
Let both $\neg A$ and $\neg B$ are satisfied. Therefore, $x$ is not a stationary point for $f$ over $X$ and the point $t=0$ is not a global minimizer of $\ph$ over $X(x,y)$. Hence, there exists a direction $d$ such that $\ld f(x,d)<0$, and there exists $\bar t\in X(x,y)$ with $\ph(\bar t)<\ph(0)$. Therefore, $f(\bar y)<f(x)$, where $\bar y=x+\bar t(y-x)$. It follows from upper semicontinuity of $f$ that there exists $\tau>0$ with $f(u)<f(x)$ where $u=\bar y-\tau d$. Denote by $z(s)=x+s(\bar y-x)$, $0\le s\le 1$,  $v(s)$ point of intersection between the straight line passing through $u$ and $z(s)$ and the ray $\{x+sd\mid s\in (0,\infty)\}$. Let $v(s)=x+\alpha(s)d$. Then $\alpha(s)=s\tau/(1-s)$. According to the quasiconvexity of $f$ we have
\[
f[z(s)]\le\max\{f(u),f[v(s)]\}.
\]
\begin{center}
\begin{picture}(100,100)
\thicklines
\put(10,50){\line(1,0){100}}
\put(10,50){\line(1,1){40}}
\put(110,50){\line(-1,-1){40}}
\put(70,10){\line(-1,4){20}}
\put(0,39){$\bar y$}
\put(110,39){$x$}
\put(41,39){$z(s)$}
\put(68,2){$v(s)$}
\put(48,95){$u$}
\put(80,54){$s$}
\put(28,54){$1-s$}
\put(23,72){$\tau$}
\end{picture}\\
\text{Fig. 1}
\end{center}

Therefore
\[
\frac{f[z(s)]-f(x)}{s}\le\max\left\{\frac{f(u)-f(x)}{s},\frac{f[v(s)]-f(x) }{s}\right\}.
\]
Taking into account that $\lim_{s\to 0^+}[f(u)-f(x)]/s=-\infty$, then we obtain that
\[
\ld f(x,\bar y-x)\le\liminf_{s\to 0^+}\frac{f[x+\alpha(s)d]-f(x)}{\alpha(s)}\cdot\frac{\alpha(s)}{s}=\tau\ld f(x,d)<0.
\]
Therefore $\ld \ph(0,1)=\ld f(x,d)<0$ and $t=0$ is not a stationary point for $\ph$. We obtain from here that $\neg C$ is satisfied.
\end{proof}

\begin{lemma}\label{pr1}
Let $x$ and $y$ be arbitrary points from the set $X$ and all hypothesis of Lemma \ref{lema3} be satisfied. Then $A \lor B \iff C$.
\end{lemma}
\begin{proof}
It follows from Lemma \ref{lema3} that $\neg A \land \neg B \Rightarrow \neg C$. On the other hand $A \Rightarrow C$ and $B\Rightarrow C$. Therefore, $\neg A\Leftarrow\neg C$ and $\neg B\Leftarrow \neg C$. We obtain from here that $\neg A\land\neg B\Leftarrow\neg C$. Hence
$\neg A\land\neg B\iff\neg C$. Taking  the relation between the negative statements, we conclude from the relation $\neg(\neg D)\iff D$.
$\neg(\neg A\land\neg B)\iff\neg(\neg C)$. Therefore $\neg(\neg A)\lor\neg(\neg B)\iff C$, that is $A\lor B\iff C$.
\end{proof}

Let us denote by the pointed parenthesis $\langle$ and $\rangle$ that the corresponding endpoint may or may not belong to the segment in question. The following lemma is Theorem 3.17 in page 49 in the book \cite{mar75}:

\begin{lemma}\label{lema5}
The scalar function $f$, defined on a convex set $X$, is quasiconvex if and only if for any segment $[x_1,x_2]\subset X$ one of the following three statements hold:

\begin{enumerate}
\item [(i)]
 $f$ is increasing or decreasing along $[x_1,x_2]$; 
\item [(ii)]
 $f$ is increasing along $(x_1,x_2]$, but not along $[x_1,x_2]$, or decreasing along $[x_1,x_2)$, but not along $[x_1,x_2]$; 
\item [(iii)]
there is a point $\bar x$ such that $f(x)$ is decreasing along $[x_1,\bar x\rangle$ and increasing along $\langle x_1,x_2]$, where at least one of these subsegments is closed.
\end{enumerate}
\end{lemma}

The following result is Theorem 2 in Ref. \cite{kom83}; see also \cite[Theorem 1]{kom83}. We provide another proof.
 
\begin{theorem}\label{th4}
Let $X\subseteq\R^n$ be a convex set. Suppose that $f:X\to\R$ is a radially continuous function. Then $f$ is pseudoconvex if and only if $f$ is quasiconvex and every stationary point is a global minimizer.
\end{theorem}
\begin{proof}
Let $f$ be pseudoconvex. Then it follows from Theorem \ref{th3} that $f$ is quasiconvex. It is  easy to see that every stationary point is a global minimizer.

Let $f$ be a quasiconvex and every stationary point be a global minimizer. Then the function $f$ satisfies Lemma \ref{lema5} We prove that $f$ is pseudoconvex. First, we prove that every point, which is not a global minimizer of the restriction, is not stationary for the restriction. Suppose that $t$ is a point which is not a global minimizer of the restriction $\ph$. Without loss of generality $t$ belongs to some interval where the $\ph$ is increasing. According to Lemma \ref{lema3} we conclude from the hypothesis that $t$ is not a stationary point of the restriction.

Second, we prove that the restriction is strictly monotone over the set $I\setminus\hat I$. Let $t_1$ and $t_2$ be two distinct points such that $\ph(t_1)=\ph(t_2)$. Then every $t\in(t_1,t_2)$ is a stationary point of $\ph$. It follows from Proposition \ref{pr1} that $x+t(y-x)$ is a stationary point of $f$ or $t$ is a global minimizer of $\ph$. By the hypothesis both cases imply that $t$ is a global minimizer of $\ph$. Therefore the whole interval $[t_1,t_2]$ belongs to the set of global minimizers of $\ph$. Therefore $\ph$ is strictly monotone outside the set of global minimizers. 

By Theorem \ref{th1}, taking into account both cases, the function $f$ is pseudoconvex.
\end{proof}

\begin{theorem}\label{th6}
Let $X\subseteq\R^n$ be a convex set and $f:X\to\R$  radially lower semicontinuous function. Then $f$ is   pseudoconvex on $X$ if and only if 
$f$ is semistrictly quasiconvex on $X$ and it obey the following property: for any $x,y\in X$ such that $f(y)<f(x)$ the number $t=0$ is not an stationary point of the function of one variable 
\[
\ph:X(x,y)\to\R,\quad \ph(t)=f(x+t(y-x)),
\]
where $X(x,y)=\{t\in\R\mid x+t(y-x)\in X\}$.
\end{theorem}
\begin{proof}
Let $f$ be pseudoconvex. It follows from Theorem \ref{th3} that it is semistrictly quasiconvex. By definition of the pseudoconvexity it is easy to see that $f(y)<f(x)$ implies $t=0$ is not a stationary point of $\ph$.

Suppose that $f$ is semistrictly quasiconvex and $f(y)<f(x)$ implies $t=0$ is not a stationary point of $\ph$. It follows from here that every point, which is not a global minimizer of $\ph$, is also not a stationary point. Taking into account Lemma \ref{lema4}, we prove that $\ph$ is {\em strictly} monotone on the intervals, where the global minimum of $\ph$ is not attained. Assume the contrary that there exist points $x\in X$ and $y\in X$ and $t_1\in\R$, $t_2\in\R$, $t_2<t_1$ such that $f[x+t_1(y-x)]=f[x+t_2(y-x)]$ and the global minimum of $\ph$ is not attained at $t_1$ and $t_2$. We conclude from semistrict quasiconvexity that $\ph(t)=\ph(t_1)=\ph(t_2)$ for all $t\in\R$ such that $t_2<t<t_1$
Therefore there exists $t_3$ such that $x+t_3(y-x)\in X$ and $\ph(t_3)<\ph(t_1)$. If $t_3<t_2$, then $\ld\ph(t_1,t_3-t_1)=0$, which contradicts, by Lemma \ref{lema4},  the assumption that $t_1$  is not a stationary point. It $t_3>t_1$, then $\ld\ph(t_2,t_3-t_2)=0$, which contradicts, by Lemma \ref{lema4},  the assumption that  $t_2$ is not a stationary point.
\end{proof}

\begin{theorem}\label{th7}
Let $X\subseteq\R^n$ be a convex set and $f:X\to\R$ pseudoconvex on $X$
radially lower semicontinuous function. Then $f$ is strictly pseudoconvex on $X$ if and only if it is
radially nonconstant, that is there is no a line segment, where the function is constant.
\end{theorem}
\begin{proof}
The theorem follows directly from Theorems \ref{th1} and \ref{th5}.
\end{proof}


\begin{thebibliography}{99}



\bibitem{baz79} Bazaraa, M.S., Shetty, C.M.: Nonlinear programming - theory and algorithms. John Wiley and Sons, New York (1979)

\bibitem{die81} Diewert, W.~E., Alternative characterizations of six kinds of quasiconvexity in the nondifferentiable case with applications to
nonsmooth programming, in: ``Generalized Concavity in Optimization and Economics'', S.~Schaible, W.~T.~Ziemba, eds.,
Academic Press, New York, 1981, 51--93.

\bibitem{kom83} Komlosi, S.: Some properties of nondifferentiable pseudoconvex functions, Math. Program. {\bf 26}, 232--237 (1983)

\bibitem{man65} Mangasarian, O.L.: Pseudo-convex functions. SIAM J. Control {\bf 3}, 281--290 (1965)

\bibitem{man69} Mangasarian, O.L.: Nonlinear programming. Repr. of the orig. 1969, Classics in Applied Mathematics vol. 10.  PA SIAM, Philadelphia (1994)

\bibitem{mar75} Martos, B.: Nonlinear programming - theory and methods, Akademiai Kiado, Budapest (1975) 

\bibitem{tuy64} Tuy, H.: Sur les in\'egalit\'es lin\'eaires. Colloq. Math. {\bf 13}, 107--123 (1964)




\end{thebibliography}
\end{document}